\documentclass{amsart}

\usepackage{amsmath}
\usepackage{amssymb}
\usepackage{eucal}
\usepackage{amscd}
\usepackage{delarray}
\usepackage{times}

\DeclareMathOperator{\im}{im}

\DeclareMathOperator{\id}{id}

\DeclareMathOperator{\dom}{dom}

\newcommand{\sm}{\smallsetminus}

\newcommand{\ol}{\overline}

\renewcommand{\epsilon}{\varepsilon}
\newcommand{\FF}{\mathcal{F}}

\newcommand{\GG}{\mathcal{G}}

\newcommand{\HH}{\mathcal{H}}
\newcommand{\VV}{\mathcal{V}}

\newcommand{\Z}{\mathbb{Z}}

\newcommand{\R}{\mathbb{R}}

\newtheorem{theorem}{Theorem}[section]
\newtheorem{lemma}[theorem]{Lemma}
\newtheorem{cor}[theorem]{Corollary}
\newtheorem{prop}[theorem]{Proposition}
\newtheorem{claim}{Claim}

\newtheorem*{teorema}{Theorem}

\theoremstyle{definition}
\newtheorem{defn}[theorem]{Definition}
\newtheorem*{rem}{Remark}
\newtheorem*{rems}{Remarks}

\newtheorem{example}[theorem]{Example}
\numberwithin{equation}{section}

\title[Riemannian foliations]{Topological description of Riemannian foliations with dense leaves}
\begin{document}

\author[Jes\'us A. \'Alvarez L\'opez]{J. A. \'Alvarez L\'opez*}
\address{$^*$Departamento de Xeometr\'\i a e Topolox\'\i a \\
Facultade de Matem\'aticas \\
Universidade de Santiago de Compostela \\
Campus Universitario Sur \\
15706 Santiago de Compostela \\Spain }
\thanks{*Research of the first author supported by DGICYT Grant
PB95-0850} 
\email{jesus.alvarez@usc.es}
\author[Alberto Candel]{A. Candel$^\dagger$}
\address{$^\dagger$ Department of Mathematics \\ CSUN \\ Northridge, CA
91330 \\ U.S.A.}

\thanks{$^\dagger$Research of the second author partially supported by
  NSF Grant DMS-0049077.}  \email{alberto.candel@csun.edu}

\date{\today}


\maketitle

\tableofcontents

\section*{Introduction}

Riemannian foliations occupy an important place in geometry. An excellent
survey is A.~Haefliger's Bourbaki seminar~\cite{haefliger}, and the book of
P.~Molino~\cite{molino} is the standard reference for Riemannian
foliations. In one of the appendices to this book, E. Ghys proposes the
problem of developing a theory of equicontinuous foliated spaces paralleling
that of Riemannian foliations; he uses the suggestive term  ``qualitative
Riemannian foliations'' for such foliated spaces.

In our previous paper~\cite{equicont}, we discussed the structure of
equicontinuous foliated spaces and, more generally, of equicontinuous
pseudogroups of local homeomorphisms of topological spaces. This concept
was difficult to develop because of
the nature of pseudogroups and the failure of having an infinitesimal
characterization of local isometries, as one does have in the
Riemannian case. These difficulties give rise to two versions of
equicontinuity: a weaker version seems to be more
natural, but a stronger version is more useful to generalize
topological properties of Riemannian foliations. Another
relevant property for this purpose is quasi-effectiveness, which is a
generalization to pseudogroups of effectiveness for group actions. In
the case of 
locally connected foliated spaces, quasi-effectiveness is equivalent to the
quasi-analyticity introduced by Haefliger \cite{Haefliger85}. For
instance, the following well-known topological properties of Riemannian
foliations were generalized to strongly equicontinuous quasi-effective
compact foliated spaces \cite{equicont}; let us remark that we also assume
that all foliated spaces are locally compact and Polish:
\begin{itemize}

\item Leaves without holonomy are quasi-isometric to one another (our original
motivation).

\item Leaf closures define a partition of the space. So the
foliated space is transitive (there is a dense leaf) if and only if it
is minimal (all leaves are dense).

\item The holonomy pseudogroup has a closure defined by using the
compact-open topology on small enough open subsets. 

\end{itemize}

In this paper, we show, in fact, that there are few ways of
constructing nice equicontinuous foliated spaces beyond Riemannian
foliations. The definition of Riemannian foliation used here is
slightly more general than usual: a foliation is called Riemannian
when its holonomy pseudogroup is given by local isometries of some
Riemannian manifold (a Riemannian pseudogroup); thus leafwise
smoothness is not required. Our main result is the following purely
topological characterization of Riemannian foliations with dense
leaves on compact manifolds.
 
\begin{teorema}\label{main theorem}
Let $(X,\FF)$ be a transitive compact foliated space. Then $\FF$ is a
Riemannian foliation if and only if $X$ is locally connected and finite
dimensional, $\FF$ is strongly equicontinuous and quasi-analytic, and the
closure of its holonomy pseudogroup is quasi-analytic. 
\end{teorema}

This theorem is a direct consequence of the corresponding result for
pseudogroups, whose proof uses the material developed in \cite{equicont} as
well as the local version of the solution of Hilbert 5th problem due to
R.~Jacoby~\cite{jacoby}. 

An earlier result in this direction was that of M. Kellum
\cite{kellum93,kellum94} who proved this property for certain
pseudogroups of uniformly Lipschitz diffeomorphisms of Riemannian
manifolds. Also, R.~Sacksteder work~\cite{sacksteder} can be
interpreted as giving a characterization of Riemannian pseudogroups of
one-dimensional manifolds. Another similar result, proved by
C.~Tarquini \cite{Tarquini}, states that equicontinuous transversely
conformal foliations are Riemannian; note that, in the case of dense
leaves, this result of Tarquini follows easily from our main theorem.

\section{Local groups and local actions}

The concept of local group and allied notions is developed in
Jacoby~\cite{jacoby}. Some of these notions are recalled in this
section, for ease of reference.

\begin{defn}
A local group is a quintuple $(G,e,\cdot,\,',\mathfrak{D})$ satisfying the
following conditions:
\begin{enumerate}
\item[(1)] $(G,\mathfrak{D})$ is a topological space;
\item[(2)] $\cdot$ is a function from a subset of $G\times G$ to $G$;
\item[(3)] $\,'$ is a function from a subset of $G$ to $G$;
\item[(4)] there is a subset $O$ of $G$ such that 
\begin{itemize}
\item[(a)] $O$ is an open neighborhood of $e$ in $G$,
\item[(b)] $O\times O$ is a subset of the domain of $\cdot$.
\item[(c)] $O$ is a subset of the domain of $\,'$,
\item[(d)] for all $a,b,c\in O$, if $a\cdot b$ and $b\cdot c$ $\in O$,
  then $(a\cdot b)\cdot c=(a\cdot b)\cdot c$.
\item[(e)] for all $a\in O$, $a'\in O$, $a\cdot e=e\cdot a=a$ and
$a'\cdot a=a\cdot a'=e$,
\item[(f)]  the map $\cdot:O\times O\to G$ is continuous,
\item[(g)] the map $\,':O\to G$ is continuous;
\end{itemize}
\item[(5)] the set $\{e\}$ is closed in $G$.
\end{enumerate}
\end{defn}

Jacoby employs the notation $\mathfrak{G}$ for the quintuple
$(G,e,\cdot,\,',\mathfrak{D})$, but here it will be simply denoted by
$G$. 

The collection of all sets $O$ satisfying condition (4) will be
denoted by $\Psi G$. This is a neighborhood base of $e\in G$; all of
these neighborhoods are symmetric with respect to the inverse
operation $(3)$.  Let $\Phi(G,n)$ denote the collection of subsets $A$
of $G$ such that the product of any collection of $\le n$ elements of
$A$ is defined, and the set $A^n$ of such products is contained in
some $O\in \Psi G$.

If $G$ is a local group, then $H$ is a subgroup of $G$ if $H\in
\Phi(G,2)$, $e\in H$, $H'=H$ and $H\cdot H=H$. 

If $G$ is a local group, then $H\subset G$ is a sub-local group of $G$
in case $H$ is itself a local group with respect to the induced
operations and topology.

If $G$ is a local group, then $\Upsilon G$ denotes the set of all
pairs $(H,U)$ of subsets of $G$ so that (1) $e\in H$; (2) $U\in \Psi
G$; (3) for all $a,b\in U\cap H$, $a\cdot b\in H$; (4) for all $c\in
U\cap H$, $c'\in H$.

Jacoby~\cite[Theorem 26]{jacoby} proves that $H\subset G$ is a
sub-local group if and only if there exists $U$ such that $(H,U)\in
\Upsilon G$.

Let $G$ be a local group and let $\Pi G$ denote the pairs $(H,U)$ so
that (1) $e\in H$; (2) $U\in \Psi G\cap \Phi(G,6)$; (3) for all
$a,b\in U^6\cap H$, $a\cdot b\in H$; (4) for all $c\in U^6\cap H$,
$c'\in H$; (5) $U^2\sm H$ is open. Given such a pair $(H,U)\in \Pi G$,
there is a (completely regular, hausdorff) topological space $G/(U,H)$
and a continuous open surjection $$T:U^2\to G/(U,H)$$ such that
$T(a)=T(b)$ if and only if $a'\cdot b\in H$ (\textit{cf.}
\cite[Theorem 29]{jacoby}).

If $(H,V)$ is another pair in $\Pi G$, then the spaces $G/(H,U)$ and
$G/(H,V)$ are locally homeomorphic in an obvious way.  Thus the
concept of coset space of $H$ is well defined in this sense, as a germ
of a topological space. The notation $G/H$ will be used in this sense;
and to say that $G/H$ has certain topological property will mean that
some $G/(H,U)$ has such property.

Let $\Delta G$ be the set of pairs $(H,U)$ such that $(H,U)\in \Pi G$
and, for all $a\in H\cap U^4$ and $b\in U^2$, $b'\cdot (a \cdot b)\in
H$. A subset $H\subset G$ is called a normal sub-local group of $G$ if
there exists $U$ such that $(H,U)\in \Delta G$. If $(H,U)\in \Delta G$
then the quotient space $G/(H,U)$ admits the structure of a local
group (see \cite[Theorem 35]{jacoby} for the pertinent details) and
the natural projection $T:U^2\to G/(H,U)$ is a local homomorphism. As
before, another such pair $(H,V)$ produces a locally isomorphic
quotient local group.

Let us recall the main results of Jacoby~\cite{jacoby} on the
structure of locally compact local groups because they will be needed
in the sequel.

\begin{theorem}[{Jacoby~\cite[Theorem~96]{jacoby}}]
\label{t:jacoby-without small subgroups}
  Any locally compact local group without small subgroups is a local Lie
  group.
\end{theorem}

In the above result, a local group without small subgroups is a local
group where some neighborhood of the identity element contains no
nontrivial subgroup.

\begin{theorem}[{Jacoby~\cite[Theorems~97--103]{jacoby}}]
\label{t:jacoby-approximated}
  Any locally compact second countable local group $G$ can be approximated by
  local Lie groups. More precisely, given $V\in \Psi G\cap\Phi(G,2)$, there
  exists $U\in \Psi G$ with $U\subset V$ and there
  exists a sequence of compact normal subgroups $F_n\subset U$ such that
  \upn{(}1\upn{)}~$F_{n+1}\subset F_n$, 
  \upn{(}2\upn{)}~$\bigcap_nF_n=\{e\}$, 
  \upn{(}3\upn{)}~$(F_n,U)\in \Delta G$, and 
  \upn{(}4\upn{)}~$G/(F_n,U)$ is a local lie group.
\end{theorem}

\begin{theorem}[{Jacoby~\cite[Theorem~107]{jacoby}}]
\label{t:jacoby-product}
  Any finite dimensional metrizable locally compact local group is
  locally isomorphic to the direct product of a Lie group and a
  compact zero-dimensional topological group.
\end{theorem}

An immediate consequence of Theorem~\ref{t:jacoby-product} is that any 
locally euclidean local group is a local Lie group, which is the local
version of Hilbert 5th problem obtained by Jacoby.
 
All local groups appearing in this paper will be assumed, or
proved, to be locally compact and second countable.

\begin{defn}
A local group $G$ is a {\em local transformation group\/} on a
subspace $X\subset Y$ if there is given a continuous map $G\times X\to
Y$, written $(g,x)\mapsto gx$, such that
\begin{itemize}
\item $ex=x$ for all $x\in X$; and
\item $g_1(g_2 x) = (g_1g_2)x$, provided both sides are defined.
\end{itemize}
This map $G\times X\to Y$ is called a {\em local action\/} of $G$ on
$X\subset Y$.
\end{defn}

The typical example of local action is the following. Let $H$ be a
sub-local group of $G$. If $(H,U)\in \Pi G$ and $T:U^2\to
G/(H,U)$ is the natural projection, then $U$ is a sub-local group of $G$
and the map $(u,T(g)) \mapsto T(u\cdot g)$ defines a local action of $U$ on
the open subspace $T(U)$ of $G/(H,U)$.

If $G$ is a local group acting on $X\subset Y$ and the action is
locally transitive at $x\in X$ in the sense that there is a
neighborhood $V\in \Psi G$ such that $V x$ includes a neighborhood of
$x$ in $X$, then there is a sub-local group $H$ of $G$ and an open
subset $U\subset G$ such that $(H,U)\in \Pi G$ and the orbit map $g\in
G\mapsto gx\in X$ induces a local homeomorphism $G/(H,U)\to X$ at $x$,
which is equivariant with respect to the action of $U$.

\begin{theorem}\label{t:main-claim}
Let $G$ be a locally compact, separable and metrizable local
group. Suppose that there is a local action of $G$ on a finite
dimensional subspace $X\subset Y$ and that the action is locally
transitive at some $x\in X$.  Fix some $(H,U)\in \Pi G$ so that the
orbit map $g\mapsto gx$ induces a local homeomorphism $G/(H,U)\to X$
at $x$.  Then there exists a connected normal subgroup $K$ of $G$ such
that $K\subset H$, $(K,U)\in\Pi G$ and $G/(K,U)$ is finite
dimensional.
\end{theorem}

\begin{proof}
This is a local version of \cite[Theorem 6.2.2]{mz}, whose proof shows the
following assertion that will be used now.

\begin{claim}\label{cl:main-claim}
Let $A$ be a locally compact, separable and metrizable topological
group, and let $B$ be a closed subgroup of $A$ such that $A/B$ is of
finite dimension and connected. Let $N_n$ be a sequence of compact
normal subgroups so that $\bigcap_nN_n=\{e\}$ and every $A/N_n$ is a
Lie group. Then there is some index $n_0$ such that the connected
component of the identity of $N_{n_0}$ is contained in $B$.
\end{claim} 

The following observation is also needed.

\begin{claim}\label{cl:embedding}
Let $A$ be a local group, let $(B,V)\in\Pi A$, let $T:A\to A/(B,V)$
denote the natural projection, and let $C$ be a compact subgroup of
$A$ contained in $V^2\cap V^6$. Then $B\cap C$ is a compact subgroup
of $C$, a map $C/(B\cap C)\to A/(B,V)$ is well defined by the
assignment $a(B\cap C)\mapsto T(a)$, and this map is an embedding.
\end{claim}

This assertion can be proved as follows. On the one hand, $B\cap C$ is
compact because $B$ is closed and $C$ compact. On the other hand,
$B\cap C$ is a subgroup of $C$ because $C$ is a subgroup, $C\subset
V^6$, and $a\cdot b\in B$ and $a'\in B$ for all $a,b\in V^6$ since
$(B,V)\in\Pi A$. The map $C/(B\cap C)\to A/(B,V)$ is well defined and
injective because $C\subset V^2$ and $T(a)=T(b)$ if and only if
$a\cdot b'\in B$ for $a,b\in V^2$. This injection is continuous
because it is induced by the inclusion $C\hookrightarrow V^2$. Thus
this map is an embedding since $C/(B\cap C)$ is compact and $A/(B,V)$
is Hausdorff.

Now, with the notation of the statement of this theorem, let $F_n$ be a
sequence of compact normal subgroups of $G$ as provided by
 Jacoby's theorem~\cite{jacoby} (quoted as
 Theorem~\ref{t:jacoby-approximated}). 
 It may be assumed that $(F_n,U)\in \Delta G$ and $F_n\subset U^2\cap
 U^6$ for all $n$.  If $K_n$ is the identity component of each $F_n$, then the
 natural quotient map $G/(K_n,U)\to G/(F_n,U)$ has zero dimensional
 fibers, because they are locally homeomorphic to the zero-dimensional
 group $F_n/K_n$. Because each $G/(F_n,U)$ is a local Lie group, it is
 finite dimensional, and thus $G/(K_n,U)$ is also finite dimensional
 (see~\cite[Ch. VII, \S 4]{hurewicz}).

By Claim~\ref{cl:embedding}, $K_1\cap H$ is a compact subgroup of $K_1$, and
there is a canonical embedding $K_1/(K_1\cap H)\to G/(H,U)$. Moreover
$K_1/(K_1\cap H)$ is connected since so is $K_1$. Then the dimension of
$K_1/(K_1\cap H)$ is less or equal than the dimension of $G/(H,U)$ by
\cite[Theorem~III~1]{hurewicz}, and thus $K_1/(K_1\cap H)$ is of finite
dimension. On the other hand, each canonical embedding
$K_1/(K_1\cap F_n)\to G/(F_n,U)$, given by Claim~\ref{cl:embedding}, realizes
$K_1/(K_1\cap F_n)$ as a compact subgroup of the local Lie group $G/(F_n,U)$
because $K_1\cap F_n$ is a normal subgroup of $K_1$. So every 
$K_1/(K_1\cap F_n)$ is a Lie group. Then, by Claim~\ref{cl:main-claim} with
$A=K_1$, $B=K_1\cap H$ and $N_n=K_1\cap F_n$, there is some index $n_0$ such
that the identity component $K$ of $F=K_1\cap F_{n_0}$ is contained in
$K_1\cap H$. This $F$ is a normal subgroup of $G$, and thus $K$ is a
connected normal subgroup of $G$. Furthermore
$(K,U),(F,U)\in\Delta G$, and
$$
\dim G/(K,U)=\dim G/(F,U)\le\dim G/(K_1,U)+\dim K_1/(K_1\cap F_{n_0})
$$
by \cite[Theorem~III~4]{hurewicz}. So
$G/(K,U)$ is of finite dimension as desired.
\end{proof}

\section{Equicontinuous pseudogroups}

A {\em pseudogroup of local transformations\/} of a topological space
$Z$ is a collection $\HH$ of homeomorphisms between open subsets of
$Z$ that contains the identity on $Z$ and is closed under composition
(wherever defined), inversion, restriction and combination of maps.
Such a pseudogroup $\HH$ is {\em generated\/} by a set $E\subset\HH$
if every element of $\HH$ can be obtained from $E$ by using the above
pseudogroup operations; the sets of generators will be assumed to be
symmetric for simplicity ($h^{-1}\in E$ if $h\in E$).  The {\em
  orbit\/} of an element $x\in Z$ is the set $\HH(x)$ of elements
$h(x)$, for all $h\in\HH$ whose domain contains $x$. These orbits are
the equivalence classes of an equivalence relation on $Z$.

Pseudogroups of local transformations are natural generalizations of
group actions on topological spaces (each group action generates a
pseudogroup). Another important example of a different nature is the
holonomy pseudogroup of a foliated space defined by a regular covering
by flow boxes \cite{CC2K,Haefliger85,Haefliger88,Hector-Hirsch}.

The study of pseudogroups can be simplified by using certain
equivalence relation introduced by Haefliger
\cite{Haefliger85,Haefliger88}. For instance, any pseudogroup of local
transformations is equivalent to its restriction to any open subset
that meets all orbits; indeed, the whole of this equivalence relation
is generated by this very basic type of examples. This concept of
pseudogroup equivalence is very important in the study of foliated
spaces because the equivalence class of the holonomy pseudogroup
depends only on each foliated space; it is independent of the choice
of a regular covering by flow boxes.

For a pseudogroup $\HH$ of local transformations of a locally compact
space $Z$, Haefliger introduced also the concept of compact
generation: $\HH$ is {\em compactly generated\/} if there is a
relatively compact open set $U$ in $Z$ meeting each orbit of $\HH$,
and such that the restriction $\GG$ of $\HH$ to $U$ is generated by a
finite symmetric collection $E\subset\GG$ so that each $g\in E$ is the
restriction of an element $\bar g$ of $\HH$ defined on some
neighborhood of the closure of the source of $g$.  This notion is
invariant by equivalences and the relatively compact open set $U$
meeting each orbit can be chosen arbitrarily. If $E$ satisfies the
above conditions, it is called a {\em system of compact generation\/}
of $\HH$ on $U$.

The concept of strong and weak equicontinuity was introduced in
\cite{equicont} for pseudogroups of local transformations of spaces
whose topology is induced by the following type of structure. Let
$\{(Z_i,d_i)\}_{i\in I}$ be a family of metric spaces such that
$\{Z_i\}_{i\in I}$ is a covering of a set $Z$, each intersection
$Z_i\cap Z_j$ is open in $(Z_i,d_i)$ and $(Z_j,d_j)$, and for all
$\epsilon>0$ there is some $\delta(\epsilon)>0$ so that the following
property holds: for all $i,j\in I$ and $z\in Z_i\cap Z_j$, there is
some open neighborhood $U_{i,j,z}$ of $z$ in $Z_i\cap Z_j$ (with
respect to the topology induced by $d_i$ and $d_j$) such that $$
d_i(x,y)<\delta(\epsilon)\Longrightarrow d_j(x,y)<\epsilon $$ for all
$\epsilon>0$ and all $x,y\in U_{i,j,z}$. Such a family is called a
{\em cover of $Z$ by quasi-locally equal metric spaces\/}.  Two such
families are called {\em quasi-locally equal\/} when their union also
is a cover of $Z$ by quasi-locally equal metric spaces.  This is an
equivalence relation whose equivalence classes are called {\em
  quasi-local metrics\/} on $Z$. For each quasi-local metric
${\mathfrak Q}$ on $Z$, the pair $(Z,{\mathfrak Q})$ is called a {\em
  quasi-local metric space\/}. Such a $\mathfrak Q$ induces a topology
on $Z$ so that, for each $\{(Z_i,d_i)\}_{i\in I}\in{\mathfrak Q}$, the
family of open balls of all metric spaces $(Z_i,d_i)$ form a base of
open sets. Any topological concept or property of $(Z,{\mathfrak Q})$
refers to this underlying topology. It was also observed in
\cite{equicont} that $(Z,{\mathfrak Q})$ is a locally compact Polish
space if and only if it is hausdorff, paracompact, separable and
locally compact.

The strongest version of equicontinuity was defined in \cite{equicont}
as follows. Let $\HH$ be a pseudogroup of local homeomorphisms of a
quasi-local metric space $(Z,{\mathfrak Q})$. Then $\HH$ is called
{\em strongly equicontinuous\/} if there exists some
$\{(Z_i,d_i)\}_{i\in I}\in{\mathfrak Q}$ and some symmetric set $S$ of
generators of $\HH$ that is closed under compositions such that, for
every $\epsilon>0$, there is some $\delta(\epsilon)>0$ so that $$
d_i(x,y)<\delta(\epsilon)\Longrightarrow d_j(h(x),h(y))<\epsilon $$
for all $h\in S$, $i,j\in I$ and $x,y\in Z_i\cap h^{-1}(Z_j\cap\im
h)$.

The condition on $S$ to be closed under compositions is precisely what
distinguishes strong and weak equicontinuity
\cite[Lemma~8.3]{equicont}. A typical choice of $S$ is the set of all
possible composites of some symmetric set of generators. In fact,
given any $S$ satisfying the condition of strong equicontinuity, it is
obviously possible to find a symmetric set of generators $E$ given by
restrictions of elements of $S$, and then the set of all composites of
elements of $E$ also satisfies the condition of strong equicontinuity.

A pseudogroup $\HH$ acting on a space $Z$ will be called {\em strongly
  equicontinuous\/} when it is strongly equicontinuous with respect to
some quasi-local metric inducing the topology of $Z$. This notion is
invariant by equivalences of pseudogroups acting on locally compact
Polish spaces \cite[Lemma~8.8]{equicont}.

A key property of strong equicontinuity is the following.

\begin{prop}[{\cite[Proposition~8.9]{equicont}}]\label{p:equicontinuous}
  Let $\HH$ be a compactly generated and strongly equicontinuous
  pseudogroup acting on a locally compact Polish quasi-local metric
  space $(Z,{\mathfrak Q})$, and let $U$ be any relatively compact
  open subset of $(Z,{\mathfrak Q})$ that meets every $\HH$-orbit.
  Suppose that $\{(Z_i,d_i)\}_{i\in I}\in{\mathfrak Q}$ satisfies the
  condition of strong equicontinuity. Let $E$ be any system of compact
  generation of $\HH$ on $U$, and let $\bar g$ be an extension of each
  $g\in E$ with $\overline{\dom g}\subset\dom\bar g$. Also, let
  $\{Z'_i\}_{i\in I}$ be any shrinking of $\{Z_i\}_{i\in I}$. Then
  there is a finite family $\VV$ of open subsets of $(Z,{\mathfrak
    Q})$ whose union contains $U$ and such that, for any $V\in\VV$,
  $x\in U\cap V$, and $h\in\HH$ with $x\in\dom h$ and $h(x)\in U$, the
  domain of $\tilde h=\bar g_n\circ\dots\circ\bar g_1$ contains $V$
  for any composite $h=g_n\circ\dots\circ g_1$ defined around $x$ with
  $g_1,\dots,g_n\in E$, and moreover $V\subset Z'_{i_0}$ and $\tilde
  h(V)\subset Z'_{i_1}$ for some $i_0,i_1\in I$.
\end{prop}

The following terminology was introduced in \cite{equicont} to study
strongly equicontinuous pseudogroups. A pseudogroup $\HH$ of local
transformations of a space $Z$ is said to be \textit{quasi-effective}
if it is generated by some symmetric set $S$ that is closed under
compositions, and such that any transformation in $S$ is the identity
on its domain if it is the identity on some non-empty open subset of
its domain. The family $S$ can be assumed to be also closed under
restrictions to open sets, and thus every map in $\HH$ is a
combination of maps in $S$ in this case. Moreover, if $\HH$ is
strongly equicontinuous and quasi-effective, then $S$ can be chosen to
satisfy the conditions of both strong equicontinuity and
quasi-effectiveness. The notion of quasi-effectiveness is invariant by
equivalences of pseudogroups acting on locally compact Polish spaces
\cite[Lemma~9.5]{equicont}. Moreover this property is equivalent to
quasi-analyticity for pseudogroups acting on locally connected and
locally compact Polish spaces \cite[Lemma~9.6]{equicont}; recall that
a pseudogroup $\HH$ is called \textit{quasi-analytic} if every $h\in
\HH$ is the identity around some $x\in \dom h$ whenever $h$ is the
identity on some open set whose closure contains $x$
\cite{Haefliger85}.

\begin{prop}[{\cite[Proposition 9.9]{equicont}}]\label{p:A,B} Let
  $\HH$ be a compactly generated, strongly equicontinuous and
  quasi-effective pseudogroup of local homeomorphisms of a locally
  compact Polish space $Z$. Suppose that the conditions of strong
  equicontinuity and quasi-effectiveness are satisfied with a
  symmetric set $S$ of generators of $\HH$ that is closed under
  compositions.  Let $A,B$ be open subsets of $Z$ such that
  $\overline{A}$ is compact and contained in $B$. If $x$ and $y$ are
  close enough points in $Z$, then
  $$ f(x)\in A\Longrightarrow f(y)\in B $$ for all $f\in S$ whose domain
  contains $x$ and $y$.
\end{prop}

Recall that a pseudogroup is called \textit{transitive} when it has a
dense orbit, and is called \textit{minimal} when all of its orbits are
dense.

\begin{theorem}[{\cite[Theorem~11.1]{equicont}}]\label{t:minimal} 
  Let $\HH$ be a compactly generated and strongly equicontinuous
  pseudogroup of local transformations of a locally compact Polish
  space $Z$. If $\HH$ is transitive, then $\HH$ is minimal.
\end{theorem}

For spaces $Y,Z$, let $C(Y,Z)$ denote the family of continuous maps
$Y\to Z$, which will be denoted by $C_{\text{\rm c-o}}(Y,Z)$ when it
is endowed with the compact-open topology.  For open subspaces $O,P$
of a space $Z$, the space $C_{\text{\rm c-o}}(O,P)$ will be considered
as an open subspace of $C_{\text{\rm c-o}}(O,Z)$ in the canonical way.

\begin{theorem}[{\cite[Theorem~12.1]{equicont}}]\label{t:closure}
  Let $\HH$ be a quasi-effective, compactly generated and strongly
  equicontinuous pseudogroup of local transformations of a locally
  compact Polish space $Z$. Let $S$ be a symmetric set of generators
  of $\HH$ that is closed under compositions and restrictions to open
  subsets, and satisfies the conditions of strong equicontinuity and
  quasi-effectiveness. Let $\widetilde{\HH}$ be the set of maps $h$
  between open subsets of $Z$ that satisfy the following property: for
  every $x\in\dom h$, there exists a neighborhood $O_x$ of $x$ in
  $\dom h$ so that the restriction $h|_{O_x}$ is in the closure of
  $C(O_x,Z)\cap S$ in $C_{\text{\rm c-o}}(O_x,Z)$. Then:
\begin{itemize}
  
\item $\widetilde{\HH}$ is closed under composition, combination and
  restriction to open sets;
  
\item every map in $\widetilde{\HH}$ is a homeomorphism around every
  point of its domain;
  
\item the maps of $\widetilde{\HH}$ that are homeomorphisms form a
  pseudogroup $\overline{\HH}$ that contains $\HH$;

\item $\overline{\HH}$ is strongly equicontinuous;
  
\item the orbits of $\overline{\HH}$ are equal to the closures of the
  orbits of $\HH$; and
  
\item $\widetilde{\HH}$ and $\overline{\HH}$ are independent of the
  choice of $S$.

\end{itemize}
\end{theorem}

If a pseudogroup $\HH$ satisfies the conditions of
Theorem~\ref{t:closure}, then the pseudogroup $\overline{\HH}$ is
called the {\em closure\/} of $\HH$.

Note that a pseudogroup $\HH$ of local transformations of a locally
compact space $Z$ is quasi-effective just when there is a symmetric
set $S$ of generators of $\HH$ that is closed under compositions and
restrictions to open subsets, and such that the restriction map
$\rho^V_W:S\cap C(V,Z)\to S\cap C(W,Z)$ is injective for all open
subsets $V,W$ of $Z$ with $W\subset V$. If moreover $Z$ is a locally
compact Polish space, and $\HH$ is compactly generated and strongly
equicontinuous, then any such $\rho^V_W$ is bijective for $V,W$ small
enough by Proposition~\ref{p:equicontinuous}. Moreover $\rho^V_W$ is
continuous with respect to the compact-open topology
\cite[p.~289]{munkres}, but it may not be a homeomorphism as shown by
the following example.

\begin{example}\label{ex:spheres}
  Let $Z$ be the union of two tangent spheres in $\R^3$, and let
  $h:Z\to Z$ be the combination of two rotations, one on each sphere,
  around the common axis and with rationally independent angles. Then
  $h$ generates a compactly generated, strongly equicontinuous and
  quasi-effective pseudogroup $\HH$ of local transformations of $Z$;
  indeed, $h$ is an isometry for the path metric space structure on
  $Z$ induced from that of $\R^3$.  Nevertheless, it is easy to see
  that the closure $\overline{\HH}$ is not quasi-effective.
\end{example}

\begin{lemma}\label{l:olHH quasi-effective} 
  Let $\HH$ be a compactly generated, strongly equicontinuous and
  quasi-effective pseudogroup of local transformations of a locally
  compact Polish space $Z$. Then $\ol{\HH}$ is quasi-effective if and
  only if there is a symmetric set $S$ of generators of $\HH$ that is
  closed under compositions and restrictions to open subsets, and such
  that $\rho^V_W:S\cap C(V,Z)\to S\cap C(W,Z)$ is a homeomorphism with
  respect to the compact-open topologies for small enough open subsets
  $V,W$ of $Z$ with $W\subset V$.
\end{lemma}

\begin{proof}
  The result follows directly by observing that, according to
  Theorem~\ref{t:closure}, $\ol{\HH}$ is quasi-effective just when
  there is some symmetric set $S$ of generators of $\HH$ that is
  closed under compositions and satisfies the following condition: for
  any sequence $h_n$ in $S$ and open non-empty subsets $V,W$ of $Z$,
  with $W\subset V\subset\dom h_n$ for all $n$, if $h_n|_W\to\id_W$ in
  $C_{\text{\rm c-o}}(W,Z)$, then $h_n|_V\to\id_V$ in $C_{\text{\rm
  c-o}}(V,Z)$.
\end{proof}

\begin{cor}\label{c:olHH quasi-analytic} 
  Let $\HH$ be a compactly generated, strongly equicontinuous and
  quasi-analytic pseudogroup of local transformations of a locally
  connected and locally compact Polish space $Z$. Then $\ol{\HH}$ is
  quasi-analytic if and only if there is a symmetric set $S$ of
  generators of $\HH$ that is closed under compositions and
  restrictions to open subsets, and such that $\rho^V_W:S\cap
  C(V,Z)\to S\cap C(W,Z)$ is a homeomorphism with respect to the
  compact-open topologies for small enough open subsets $V,W$ of $Z$
  with $W\subset V$.
\end{cor}

Finally, let us recall from \cite{equicont} certain isometrization
theorem, which states that equicontinuous quasi-effective pseudogroups
are indeed pseudogroups of local isometries in some sense.  First, two
metrics on the same set are said to be {\em locally equal\/} when they
induce the same topology and each point has a neighborhood where both
metrics are equal.  Let $\{(Z_i,d_i)\}_{i\in I}$ be a family of metric
spaces such that $\{Z_i\}_{i\in I}$ is a covering of a set $Z$, each
intersection $Z_i\cap Z_j$ is open in $(Z_i,d_i)$ and $(Z_j,d_j)$, and
the metrics $d_i,d_j$ are locally equal on $Z_i\cap Z_j$ whenever this
is a non-empty set. Such a family will be called a {\em cover of $Z$
  by locally equal metric spaces\/}. Two such families are called {\em
  locally equal\/} when their union also is a cover of $Z$ by locally
equal metric spaces. This is an equivalence relation whose equivalence
classes are called {\em local metrics\/} on $Z$. For each local metric
${\mathfrak D}$ on $Z$, the pair $(Z,{\mathfrak D})$ is called a {\em
  local metric space\/}.  Observe that every metric induces a unique
local metric in a canonical way.  In turn, every local metric
canonically determines a unique quasi-local metric. Note also that
local metrics induced by metrics can be considered as germs of metrics
around the diagonal. Moreover a local or quasi-local metric is induced
by some metric if and only if it is hausdorff and paracompact
\cite[Theorems~13.5 and~15.1]{equicont}.

Now, a local homeomorphism $h$ of a local metric space $(Z,{\mathfrak
  D})$ is called a {\em local isometry\/} if there is some
$\{(Z_i,d_i)\}_{i\in I}\in{\mathfrak D}$ such that, for $i,j\in I$ and
$z\in Z_i\cap h^{-1}(Z_j\cap\im h)$, there is some neighborhood
$U_{h,i,j,z}$ of $z$ in $Z_i\cap h^{-1}(Z_j\cap\im h)$ so that
$d_i(x,y)=d_j(h(x),h(y))$ for all $x,y\in U_{h,i,j,z}$. This
definition is independent of the choice of the family
$\{(Z_i,d_i)\}_{i\in I}\in\mathfrak D$.  Then the isometrization
theorem is the following.

\begin{theorem}[{\cite[Theorem~ 15.1]{equicont}}]\label{t:isometrization}
  Let $\HH$ be a compactly generated, quasi-effective and strongly
  equicontinuous pseudogroup of local transformations of a locally
  compact Polish space $Z$. Then $\HH$ is a pseudogroup of local
  isometries with respect to some local metric inducing the topology
  of $Z$.
\end{theorem}

\section{Riemannian pseudogroups}

\begin{defn}\label{d:Riemannian pseudogroup}
  A pseudogroup $\HH$ of local transformations of a space $Z$ is
  called a \textit{Riemannian pseudogroup} if $Z$ is a 
  Hausdorff paracompact $C^\infty$-manifold and all maps in $\HH$ are local
  isometries with respect to some Riemannian metric on $Z$.
\end{defn}

\begin{example}\label{e:lie}
  Let $G$ be a local Lie group, $G_0\subset G$ a compact subgroup.
  Then the canonical local action of some neighborhood of the identity
  in $G$ on some neighborhood of the identity class in $G/G_0$
  generates a transitive Riemannian pseudogroup. In fact, since $G_0$
  is compact, there is a $G$-left invariant and $G_0$-right invariant
  Riemannian metric on some neighborhood of the identity in $G$, which
  induces a $G$-invariant Riemannian metric on some neighborhood of
  the identity class in $G/G_0$.  With more generality, if
  $\Gamma\subset G$ a dense sub-local group, then the canonical local
  action of some neighborhood of the identity in $\Gamma$ on some
  neighborhood of the identity class in $G/G_0$ generates a transitive
  Riemannian pseudogroup.  Moreover this Riemannian pseudogroup is
  complete in the sense of \cite{Haefliger85}. It is well known that
  any transitive complete Riemannian pseudogroup is equivalent to a
  pseudogroup of this type, which follows from the pseudogroup version
  of Molino description of Riemannian foliations.
\end{example}

The pseudogroup version of the main result of this paper is the
following topological characterization of transitive compactly
generated Riemannian pseudogroups.

\begin{theorem}\label{t:Riemannian}
  Let $\HH$ be a transitive, compactly generated pseudogroup of local
  transformations of a locally compact Polish space $Z$. Then $\HH$ is a
  Riemannian pseudogroup if and only if $Z$ is locally
  connected and finite dimensional, $\HH$ is strongly equicontinuous
  and quasi-analytic, and $\ol{\HH}$ is quasi-analytic.
\end{theorem}

\begin{rem}
  The closure $\ol{\HH}$ of $\HH$ exists by virtue of
  Theorem~\ref{t:closure}, because the space $Z$ is locally connected,
  hence the pseudogroup $\HH$ is quasi-effective because it is
  quasi-analytic \cite[Lemma 9.6]{equicont}.
\end{rem}

The following is a direct consequence of the above theorem.

\begin{cor}
  Let $\HH$ be a compactly generated, strongly equicontinuous and
  quasi-analytic pseudogroup of local transformations of a locally
  compact Polish space $Z$. Then the $\HH$-orbit closures are
  $C^\infty$ manifolds if and only if they are locally connected and
  finite dimensional, and the induced pseudogroup $\ol{\HH}$ is
  quasi-analytic on them.
\end{cor}

\begin{proof}
  This follows from Theorem~\ref{t:Riemannian} because the closure of
  $\HH$ acting on the closure of an orbit is equivalent to a
  pseudogroup like Example~\ref{e:lie}.
\end{proof}

The proof of Theorem~\ref{t:Riemannian} will be given in the next
section; in the interim, some examples illustrating the necessity of
several hypotheses are described.

\begin{example}
  Let $Z$ be the product of countably infinitely many circles. This is
  a compact, locally connected Polish group which acts on itself by
  translations in an equicontinuous way. Let $\Z\to Z$ be an injective
  homomorphism with dense image. Then the action of $\Z$ on $Z$
  induced by this homomorphism is minimal and equicontinuous, and so
  it generates a minimal, quasi-analytic and equicontinuous
  pseudogroup, which is not Riemannian because $Z$ is of infinite
  dimension.
\end{example}

\begin{example} 
  Let $Z$ be the set of $p$-adic numbers $x\in \mathbb{Q}_p$ with
  $p$-adic norm $|x|_p\le 1$. Then the operation $x\mapsto x+1$ defines
  an action of $\Z$ on $Z$ which is minimal and equicontinuous (it
  preserves the $p$-adic metric on $Z$). Thus it generates a minimal,
  quasi-analytic and equicontinuous pseudogroup, which is not
  Riemannian because $Z$ is zero-dimensional.
\end{example}

\begin{example}
  A related example is as follows. Let $Z$ be the standard cantor set
  in $[0,1]\subset \R$ together with all integer translates. Then
  there is a pseudogroup $\HH$ acting on $Z$ which is generated by
  translations of the line which locally preserve $Z$. In fact, $\HH$
  is a pseudogroup of local isometries for two geometrically distinct
  metrics, the euclidean and the dyadic.
\end{example}

%
%

\begin{example}
  The previous example can be generalized, replacing $Z$ by the
  universal Menger curve \cite[Ch. 15]{menger}. This space $Z$ (to be
  precise, a modification of it) is constructed as an invariant set of
  the pseudogroup of local homeomorphisms of $\R^3$ generated by the
  map $f(x)=3x$ and the three unit translations parallel to the
  coordinate axes.  There is a pseudogroup acting on $Z$ generated by
  euclidean isometries which locally preserve $Z$. It is fairly easy
  to see that such a pseudogroup is minimal, quasi-analytic and
  equicontinuous. Moreover $Z$ is locally connected and of dimension
  one.  However, this pseudogroup is not compactly generated.
\end{example}

\section{Equicontinuous pseudogroups and Hilbert's 5th problem}

This section is devoted to the proof of Theorem~\ref{t:Riemannian}.
The ``only if'' part is obvious, so it is enough to show the ``if''
part, which has essentially two steps. In the first one, a local group
action on $Z$ is obtained as the closure of the set of elements of
$\HH$ which are sufficiently close to the identity map on an
appropriate subset of $Z$. This construction follows
Kellum~\cite{kellum93}. The second step invokes the theory behind the
solution to the local version of Hilbert's 5th problem in order to
show that the local group is a local Lie group, and thus this local
action is isometric for some Riemannian metric if its isotropy
subgroups are compact.  So $\HH$ is proved to be Riemannian by showing
that it is of the type described in Example~\ref{e:lie}.

By Theorem~\ref{t:isometrization}, there is a local metric structure
$\mathfrak D$ on $Z$ with respect to which the elements of $\HH$ are
local isometries. Take any $\{(Z_i,d_i)\}_{i\in I}\in{\mathfrak D}$
satisfying the condition of strong equicontinuity.  Let $U$ be a
relatively compact non-trivial open subset of $Z$, and $\VV$ a family
of open subsets which cover $U$ as in
Proposition~\ref{p:equicontinuous}. Let $V$ be an element of $\VV$
which meets $U$, which is assumed to be contained in $Z_{i_0}$ for
some $i_0\in I$, and let $D\subset V$ be an open connected subset with
compact closure also contained in $V$. According to
Proposition~\ref{p:equicontinuous}, if $h\in \HH$ is such that $\dom
h\subset D$ and $h(D) \cap U\ne \emptyset$, then there exists an
element $\tilde{h}\in \HH$ which extends $h$ and whose domain contains
$V$. Moreover, as $\HH$ is quasi-analytic and $D$ connected, such
extension $\tilde{h}$ is unique on $D$. In particular, such $h$ admits
a unique extension to a homeomorphism of $\ol{D}$ onto its image.

Under the current hypothesis, the completion $\ol{\HH}$ of $\HH$ is a
quasi-analytic pseudogroup of transformations of $Z$ whose action on
$Z$ has a single orbit.  Let $\ol{\HH}_D$ be the collection of
homeomorphisms $h|_D$ with $h\in \ol{\HH}$ an element whose domain
contains $D$. Let $D'\subset D$ be a connected, compact set with
non-empty interior, and let
$$
\ol{\HH}_{DD'} =\left\{ h\in \ol{\HH}_D\mid h(D')\cap D'\ne
\emptyset\right\}\;.
$$ 
By the strong equicontinuity of $\ol{\HH}$ and
Proposition~\ref{p:A,B}, the set $D'$ can be chosen so that all the
translates $h(D')$, $h\in \ol{\HH}_{DD'}$, are contained in a fixed
compact subset $K$ of $D$.  Once this choice of $D'$ is made, let
$G=\ol{\HH}_{DD'}$ be the resulting space.

The space $G$ is endowed with the compact open topology as a subset of
$C(D,Z)$. Every element of $G$ is actually defined on $V$, hence on
$\ol{D}$, and so the compact open topology can be described by the
supremum metric given by
$$ d(g_1,g_2)=\sup_{x\in D} d_{i_0}(g_1(x), g_2(x))\;,$$
where $d_{i_0}$ is the distance function on $Z_{i_0}\subset Z$ as above.

\begin{lemma}
  Endowed this topology, $G$ is a compact space.
\end{lemma}

\begin{proof}
  It has to be shown that any sequence $g_n$ of elements of $G$ has a
  convergent subsequence. By equicontinuity, $g_n$ may be assumed to
  be made of elements of $\HH$. By Proposition~\ref{p:equicontinuous}
  and the definition of $G$, each $g_n$ can be extended to a
  homeomorphism whose domain contains $V$. According to
  Theorem~\ref{t:closure}, the sequence $g_n$ converges uniformly on
  $D$ to a map $g\in\widetilde{\HH}$.  It needs to be shown that
  $g:D\to g(D)$ is a homeomorphism and that it satisfies $g(D')\cap
  D'\ne\emptyset$.
  
  To verify this last condition, note that, for each $n$, there exists
  $x_n\in D'$ such that $g_n(x_n)\in D'$, by the definition of $G$.
  Since $D'$ is compact, it may be assumed that $x_n\to x\in D'$,
  yielding $g_n(x_n)\to g(x)\in D'$ since $g_n\to g$ uniformly on $D$.
  Thus, $g(D')\cap D'\ne\emptyset$.
  
  Each $g_n:D\to g_n(D)\subset V$ is a homeomorphism. Thus, by
  Proposition~\ref{p:equicontinuous}, there are
  maps $h_n\in\HH$ defined on $V$ such that $h_n\circ g_n=\id$ on
  $g_n(D)$.  If $g$ fails to be a homeomorphism on $D$, then there are
  points $x,y\in D$ with $d_{i_0}(x,y) >0$ and $g(x)=g(y)=z$. The map
  $g$ is a homeomorphism around each point of $D$, as Theorem~\ref{t:closure}
  shows. Thus there are disjoint neighborhoods $O_x$ and $O_y$ of $x$
  and $y$, respectively, such that $g$ maps each of them
  homeomorphically onto a neighborhood $W$ of $z$. Since the sequences
  $g_n(x), g_n(y)$ both converge to $z$, they may be assumed to be
  contained in $W$.  Furthermore, perhaps further shrinking $W$, the
  restrictions $h_n|_W$ form an equicontinuous family, which therefore
  converges to a map $h$ which inverts $g$ on $W$. This situation
  contradicts the fact that $h_n(g_n(x))$ and $h_n(g_n(y))$ do not have
  the same limit.
\end{proof}

The following lemma is similar to the corresponding one in
Kellum~\cite{kellum93}.

\begin{lemma}
  The space $G$, endowed with the compact-open topology and the
  operations just described, is a locally compact local group.
\end{lemma}

\begin{proof}
  Let $g_1,g_2$ be two elements of $G$. Then the composition $g_1\circ
  g_2$ is defined on $D'$ because $g_1(D')\subset D$. Therefore there
  exists $h\in \ol{\HH}_D$ which extends $g_1\circ g_2$. By
  quasi-analyticity of $\ol{\HH}$, this extension is unique and thus
  it defines a map $(g_1,g_2) \mapsto g_1\cdot g_2$ from $G\times G$
  into $\ol{\HH}$. 
%
%
%
%
  If $g_1,g_2$ are sufficiently close to the identity of $D$ in the
  compact open topology of $C(D,Z)$, then also $g_1\cdot g_2\in G$.
  
  The existence of a unique identity element $e$ for $G$, as well as
  the existence of an inverse operation on $G$,  is
  proved analogously.
  
  Finally, by Corollary~\ref{c:olHH quasi-analytic} and the
  quasi-analyticity of $\ol{\HH}$, it is easy to see that the local
  group multiplication and inverse map are continuous with the compact
  open topology on $G$.
\end{proof}

%
%
%
%
%
%
%

\begin{rems}
  The quasi-analyticity for $\ol{\HH}$ was used in this proof.  Note
  that it would be needed even to prove, in a similar way, that
  $\Gamma$ is a local group. The final section of the paper discusses
  the necessity of this condition in some more detail.

  By Theorem~\ref{t:isometrization}, we can assume that all elements
  of $G$ are isometries with respect to $d_{i_0}$. Then it easily
  follows that the above distance $d$ on $G$ is left invariant.
\end{rems}

The following lemma follows easily; \textit{cf.} \cite{kellum93}.

\begin{lemma}
  The map $G\times D'\to D$ defined by $(g,x)\mapsto g(x)$ makes $G$
  into a local group of transformations on $D'\subset D$.
\end{lemma}

Let $\Gamma=\HH\cap G$, which is a finitely generated dense sub-local
group of $G$. The following is a direct consequence of the minimality
of $\HH$.

\begin{lemma}\label{l:Gamma}
  $\HH$ is equivalent to the pseudogroup on any non-empty open subset
  of $D'$ generated by the local action of $\Gamma$ on $D'\subset Z$.
\end{lemma}

Let $x_0$ be a point in the interior of $D'$, which will remain fixed
from now on. Note that, by construction, all elements of $G$ are
defined at $x_0$.  Let $\phi:G\to D$ be the orbit map given by
$\phi(g)=g(x_0)$. This map is continuous because the action is
continuous.

\begin{lemma}
  The image of the orbit map $\phi$ contains a neighborhood of $x_0$.
\end{lemma}
\begin{proof}
  $\HH$ is minimal by Theorem~\ref{t:minimal}, and thus the space $Z$
  is locally homogeneous with respect to the pseudogroup $\ol{\HH}$ by
  Theorem~\ref{t:closure}. More precisely,
  Proposition~\ref{p:equicontinuous} and Theorem~\ref{t:closure} show
  that, given $x\in D'$, there exists $h\in\overline{\HH}$ with domain
  $\dom h= D$ such that $h(x_0)=x$. Since both $x,x_0\in D'$, it
  follows that $h\in G$. The statement follows immediately from this.
\end{proof}

Let $G_0$ denote the collection of elements $g\in G$ such that
$g(x_0)=x_0$.

\begin{lemma}\label{l:no-normal-subgroup}
  The set $G_0$ is a compact subgroup of $G$.
\end{lemma}

\begin{proof}
  First, $G_0$ is compact because, being the stabilizer of a point, it
  is a closed subset of $G$ and $G$ is a compact hausdorff space.

  Second, it follows from the definitions of $G$ and of its group
  multiplication that the product of two elements of $G_0$ is defined
  and belongs to $G_0$, and likewise the inverse of every element.
  More precisely, if $g_1,g_2\in G_0$, then $g_1\circ g_2$ is an
  element of $\ol{\HH}$ which fixes $x_0$, hence $g_1\circ g_2(D')\cap
  D'\ne \emptyset$.
\end{proof}

In the special case of the group $G_0$ which stabilizes $x_0$
considered here, the equivalence relation $\sim$ on $G$ used to define
a representative coset space of $G_0$ can also be defined as $h\sim g$
if and only if $h(x_0)=g(x_0)$.

\begin{lemma}\label{l:psi}
  The orbit map $\phi:G\to Z$ induces a map $\psi:G/G_0\to Z$ which is
  a homeomorphism of a neighborhood of the identity class in $G/G_0$
  onto a neighborhood of $x_0$ in $Z$. 
\end{lemma}

\begin{proof}
  This follows directly from the preceding discussion on coset spaces
  and the finite dimensionality of $Z$.
\end{proof}

\begin{cor}\label{c:Gamma}
  $\HH$ is equivalent to the pseudogroup induced by the canonical local
  action of some neighborhood of the identity in $\Gamma$ on some
  neighborhood of the identity class in $G/G_0$.
\end{cor}

\begin{proof}
  This follows from Lemmas~\ref{l:Gamma} and~\ref{l:psi}.
\end{proof}

\begin{cor}\label{c:fin-dim}
  $G/G_0$ is finite dimensional.
\end{cor}

\begin{proof}
  This follows directly from Lemma~\ref{l:psi} and the finite
  dimensionality of $Z$.
\end{proof}
 
\begin{lemma}\label{l:nonormal}
  The group $G_0$ contains no non-trivial normal sub-local group of
  $G$.
\end{lemma}
\begin{proof}
  If $N\subset G$ is a normal sub-local group contained in $G_0$ and
  $n\in N$, then for each $g$ in a suitable neighborhood of $e$ in $G$
  there is some $n'\in N$ so that
  $$
  n\,\phi(g)=n\,g(x_0)=g\,n'(x_0)=g(x_0)=\phi(g)\;.
  $$ 
  Thus $n$ acts trivially on a neighborhood of $x_0$ in $D'$. This is
  possible only if $n=e$, because $\ol{\HH}$ is quasi-analytic.
\end{proof}

\begin{lemma}
  The local group $G$ is finite dimensional.
\end{lemma}

\begin{proof}
  By Corollary~\ref{c:fin-dim}, $G/G_0$ is finite dimensional. By
  Theorem~\ref{t:main-claim}, there exists a compact normal subgroup
  $(K,U)$ in $\Delta G$ such that $K\subset G_0$ and $G/(K,U)$ is finite
  dimensional. Lemma~\ref{l:nonormal} implies that $K$ is
  trivial;  thus $G$ is finite dimensional because it is locally
  isomorphic to $G/K$.
\end{proof}

Finally, since $G_0$ is a compact subgroup of $G$, the following
finishes the proof of Theorem~\ref{t:Riemannian} according to
Example~\ref{e:lie} and Corollary~\ref{c:Gamma}.

\begin{lemma}
  The group $G$ is a local Lie group.
\end{lemma}

\begin{proof}
  This is a local version of \cite[Theorem 73]{pontriaguin}.  By
  Theorem~\ref{t:jacoby-without small subgroups}, it is enough to show
  that $G$ has no small subgroups. The local group $G$ is finite
  dimensional and metrizable, so Theorem~\ref{t:jacoby-product}
  implies that there is a neighborhood $U$ of $e$ in $G$ which
  decomposes as the direct product of a local Lie group $L$ and a
  compact zero-dimensional normal subgroup $N$.  Then $P=N\cap G_0$ is
  a normal subgroup of $G_0$, and $G_0/P$ is a Lie group because it is
  a group which is locally isomorphic to the local Lie group $G/(N,U)$
  (\textit{cf}. \cite[Theorem 36]{jacoby}).
  
  Furthermore, since $N$ is zero-dimensional, so is $P$. Thus there exists a
  neighborhood $V$ of $e$ in $G_0$ which is the direct product of a connected
  local lie group $M$ and the normal subgroup $P$. It may be assumed that
  $V\subset U$. Since $M$ is connected and $N$ is zero-dimensional, it
  follows that $M\subset L$.  
%
%
  Summarizing, there is a local isomorphism between $G$ and the direct
  product $L\times N$, which restricts to a local isomorphism of $G_0$
  to $M\times P$. Therefore, there exists a neighborhood of the class
  $G_0$ in $G/G_0$ which is homeomorphic to a neighborhood of the
  class of the identity in the product $L/M\times N/P$. It follows
  that a neighborhood of $x_0$ in $Z$ is homeomorphic to the product
  of an euclidean ball and an open subspace $T\subset N/P$. Since $Z$
  is by assumption locally connected and $N/P$ zero-dimensional, it
  follows that $T$ is finite, and hence that $N/P$ is a discrete
  space.
%
%
%
  So $P$ is an open subset of $N$ and thus there exists a neighborhood
  $W$ of $e$ in $G$ such that $W\cap P=W\cap N$. By the local
  approximation of Jacoby (Theorem~\ref{t:jacoby-approximated}), there
  exists a compact normal subgroup $K\subset W$ such that $G/K$ is a
  local Lie group.  Then $G_0$ contains $P\cap K$, which is equal to
  the normal subgroup $N\cap K$ of $G$ because $K\subset W$. Thus, by
  Lemma~\ref{l:nonormal}, $N\cap K$ is trivial. On the other hand,
  $N/(N\cap K)$ is a zero-dimensional lie group, hence $N\cap K$ is
  open in $N$. It follows that $N$ is finite, and thus that $G$ is a
  local lie group.
\end{proof}

\section{A description of transitive, compactly generated, strongly
equicontinuous and quasi-effective pseudogroups}

The following example is slightly more general than
Example~\ref{e:lie}.

\begin{example}\label{e:locally homogeneous} 
  Let $G$ be a locally compact, metrizable and separable local group,
  $G_0\subset G$ a compact subgroup, and $\Gamma\subset G$ a dense
  sub-local group.  Suppose that there is a left invariant metric on
  $G$ inducing its topology. This metric can be assumed to be also
  $G_0$-right invariant by the compactness of $G_0$.  Then the
  canonical local action of $\Gamma$ on some neighborhood of the
  identity class in $G/G_0$ induces a transitive strongly
  equicontinuous and quasi-effective pseudogroup of local
  transformations of a locally compact Polish space. In fact, this is
  a pseudogroup of local isometries in the sense of \cite{equicont}.
\end{example}

The proof of the following theorem is a straightforward adaptation of
the first part of the proof of Theorem~\ref{t:Riemannian}, by using
quasi-effectiveness instead of quasi-analyticity.

\begin{theorem}\label{t:locally homogeneous} 
  Let $\HH$ be a transitive, compactly generated, strongly
  equicontinuous and quasi-effective pseudogroup of local
  transformations of a locally compact Polish space, and suppose that
  $\ol{\HH}$ is also quasi-effective. Then $\HH$ is equivalent to a
  pseudogroup of the type described in Example~\ref{e:locally
    homogeneous}.
\end{theorem}

The study of compact generation for the pseudogroups of
Example~\ref{e:locally homogeneous} is very delicate
\cite{meigniez92}. But those pseudogroups are obviously complete, and
it seems that compact generation could be replaced by completeness in
Theorem~\ref{t:locally homogeneous}, obtaining a better result. This
would require the generalization of our work \cite{equicont} to
complete strongly equicontinuous pseudogroups.

\section{Quasi-analyticity of pseudogroups}

The most elusive of the hypotheses of Theorem~\ref{t:Riemannian} is
that concerning the quasi-analyticity of the closure of a
quasi-analytic pseudogroup. We do not have an example of a pseudogroup
$\HH$ as in the main theorem whose closure fails to be quasi-analytic.
Thus this section offers some examples and observations relevant to
this problem.

It is quite easy to see that the definition of length space has a
local version as described in \cite[Section 13]{equicont}, and that
the two theorems of Section 15 of such paper are also available for
local length spaces.

There are many examples of metric spaces which admit actions of
pseudogroups of isometries which are not quasi-analytic.  The
following two are examples of real trees (see Shalen~\cite{shalen}).

\begin{example}
  Let $X=\R^2$ be endowed with the metric given by
  $$d((x_1,y_1),(x_2,y_2))=
\begin{cases}
  |y_1|+|x_1-x_2|+|y_2| & \text{if $x_1\ne x_2$}\\
  |y_1-y_2| & \text{if $x_1=x_2$}
\end{cases}
  $$
  Given a subset $F$ of the real axis there is an isometry $f$ of $X$
  given by
  $$
  f(x,y)=
  \begin{cases}
    (x,y) &\text{if $x\in F$} \\
    (x,-y) & \text{if $x\notin F$}
  \end{cases}
  $$

  This family of isometries $f$ forms a normal subgroup of the group
  of isometries of $X$. Thus this group is not quasi-analytic,
  although $X$ is not locally compact.
\end{example}

\begin{example}
  Let $X=\R*\R$ be the free product of two copies of $\R$. Then $X$
  has the structure of a real tree, and is a homogeneous space with
  respect to its group of isometries. It is not quasi-analytic.
\end{example}


\begin{defn}
  A local length space $X$ is analytic at a point $x\in X$ if the
  following holds: if $\gamma,\gamma'$ are geodesic arcs (parametrized
  by arc-length) defined on an interval about $0\in \R$, such that
  $\gamma(0)=\gamma'(0)=x$ and that $\gamma=\gamma'$ on some interval $(-a,0]$,
  then they have the same germ at $0$. The space $X$ is analytic if it
  is analytic at every point.
\end{defn}

\begin{example} 
  A Riemannian manifold is an analytic length space. Real trees with
  many branches, as in the above examples, are not analytic.

  In relation to Theorem~\ref{t:Riemannian}, if a local length space
  is known to be analytic at one point and admits a transitive action
  of a pseudogroup of local isometries, then it is analytic.
%
\end{example}

%
%
%

\begin{prop}
  Let $\HH$ be a pseudogroup of local isometries of an analytic local
  length space $X$. Then $\HH$ is quasi-analytic.
\end{prop}
\begin{proof}
  If $\HH$ is not quasi-analytic, then, by definition, there exists an
  element $f$ of $\HH$, an open subset $U$ in $\dom(f)$ such that
  $f|U=\id$, and a point $x_0$ in the closure of $U$ such that $f$ is
  not the identity in any neighborhood of $x_0$. Therefore, there is a
  sequence of points $x_n$ converging to $x_0$ such that $f(x_n)\ne
  x_n$. If $n$ is sufficiently large, then there is a geodesic arc
  contained in the domain of $f$ and joining $x_n$ and a point $y\in
  U$. This geodesic arc is mapped by $f$ to a distinct geodesic arc
  having the same germ at one of its endpoints.
\end{proof}

The following example shows  that the converse is false.

\begin{example}
  Let $X$ be the euclidean space $\R^2$ endowed with the metric
  induced by the supremum norm $\|(x,y)\}=\max \{ |x|,|y|\}$. Then $X$
  is a length space which is not locally analytic. Indeed, if $f:I\to
  R$ is a function such that $|f(s)-f(t)|\le |s-t|$ for all $s,t\in
  I$, then $t\in I\mapsto (t,f(t))\in X$ is a geodesic. However, every
  local isometry is locally equal to a linear isometry, hence the
  pseudogroup of local isometries is quasi-analytic.
\end{example}

%
%


\begin{thebibliography}{00}

\bibitem{equicont}
 J. A. \'Alvarez L\'opez and A. Candel, {\em Equicontinuous foliated
   spaces}, Math. Z., to appear.

\bibitem{CC2K} A. Candel and L. Conlon, {\em Foliations I}, Graduate
  Studies in Mathematics {\bf 23}, American Mathematical Society,
  Providence, R.I., 2000.

\bibitem{menger} L. Blumenthal and K. Menger, {\em Studies in
    Geometry}, W. H. Freeman and Co., San Francisco, Calif., 1970.

\bibitem{Haefliger85} A. Haefliger, {\em Pseudogroups of local
    isometries\/}, Differential Geometry (Santiago de Compostela
  1984), L.A. Cordero, ed., Research Notes in Math. {\bf 131}, Pitman
  Advanced Pub. Program, Boston, 1985, pp.~174--197.

\bibitem{Haefliger88} A. Haefliger, {\em Leaf closures in Riemannian
    foliations}, A F\^{e}te on Topology, Y. Matsumoto, T. Mizutani and
  S. Morita, eds., Academic Press, New York, 1988, pp.~3--32.

\bibitem{haefliger} A. Haefliger, {\em Feuilletages riemanniens\/},
  Séminaire Bourbaki, Vol. 1988/89.  Astérisque No. 177-178, (1989),
  Exp. No. 707, pp.~183--197.

\bibitem{Hector-Hirsch} G. Hector and U. Hirsch, {\em Introduction to
    the Geometry of Foliations, Parts~{A} and~{B}\/}, Aspects of
  Mathematics, vols. {\bf E1} and {\bf E3}, Friedr. Vieweg and Sohn,
  Braunschweig, 1981 and 1983.

\bibitem{hurewicz} W. Hurewicz and H. Wallman, {\em Dimension Theory},
  Princeton University Press, Princeton, N.J., 1941.

\bibitem{jacoby} R. Jacoby, {\em Some theorems on the structure of
    locally compact local groups}, Annals of Math. {\bf 66} (1957),
  36--69.

\bibitem{kellum93} M. Kellum, {\em Uniformly quasi-isometric
    foliations}, Ergodic Theory Dynamical Systems {\bf 13} (1993),
  101--122.

\bibitem{kellum94} Kellum, M., {\em Uniform Lipschitz distortion,
    invariant measures and foliations}.  Geometric study of foliations
  (Tokyo, 1993), World Sci. Publishing, River Edge, N.~J., 1994,
  pp.~313--326.

\bibitem{meigniez92} G.~Meigniez, {\em Sous-groupes de g\'eneration
    compacte des groupes de Lie r\'esolubles\/}, Pr\'epublications
  Math\'ematiques de l'Universit\'e Paris~VII, N.~33, 1992.


\bibitem{molino} P. Molino, {\em Riemannian Foliations} (with
  appendices by G. Cairns, Y. Carri\`ere, E. Ghys, E. Salem and V.
  Sergiescu), Progress in Mathematics, vol. {\bf 73}, Birkh\"auser,
  Boston, 1988.

\bibitem{munkres} J.~R. Munkres, {\em Topology: a First Course\/},
  Prentice-Hall, Inc., Englewood Cliffs, N.J., 1975.

\bibitem{mz} D. Montgomery and L. Zippin, {\em Topological
    Transformation Groups}, New York, 1955.

\bibitem{pontriaguin} L.~S. Pontriaguin, {\em Grupos Continuous\/},
  Editorial Mir, Mosc\'u, 1978.

\bibitem{sacksteder} R. Sacksteder, {\em Foliations and pseudogroups},
  American J. of Math. {\bf 87} (1965), 79--102.

\bibitem{shalen} P. Shalen, {\em Dendrology of groups\/}, Essays in
  Group Theory, S. M. Gersten, Ed., Math. Sci. Res. Inst. Publ., vol.
  {\bf 8}, Springer, New York, 1987, pp.~265--319.

\bibitem{Tarquini} C.~Tarquini, {\em Feuilletages conformes et
    feuilletages \'equicontinus\/}, preprint, 2001.

\end{thebibliography}
\end{document}